\documentclass[11pt]{article}
\usepackage[numbers,sort&compress]{natbib}
\usepackage{enumerate}
\usepackage{amscd}
\usepackage{amsmath}
\usepackage{latexsym}
\usepackage{amsfonts}
\usepackage{amssymb}
\usepackage{amsthm}
\usepackage{verbatim}
\usepackage{mathrsfs}
\usepackage{enumerate}
\usepackage{hyperref}

\theoremstyle{plain}\newtheorem{definition}{Definition}[section]
\theoremstyle{definition}\newtheorem{theorem}{Theorem}[section]
\theoremstyle{plain}\newtheorem{lemma}[theorem]{Lemma}
\theoremstyle{plain}\newtheorem{coro}[theorem]{Corollary}
\theoremstyle{plain}\newtheorem{prop}[theorem]{Proposition}
\theoremstyle{remark}\newtheorem{remark}{Remark}[section]
\usepackage{xcolor}
\newcommand{\wblue}[1]{\textcolor{black}{#1}}

\newcommand{\Div}{\mathrm{div}\,}
\newcommand{\B}{\Big}

\newcommand{\be}{\begin{equation}}
\newcommand{\ee}{\end{equation}}
\newcommand{\ba}{\begin{aligned}}
	\newcommand{\ea}{\end{aligned}}

\providecommand{\bysame}{\leavevmode\hbox to3em{\hrulefill}\thinspace}
\newcommand{\f}{\frac}

\newcommand{\ben}{\begin{enumerate}}
	\newcommand{\een}{\end{enumerate}}

\newcommand{\ti}{\nabla}

\newcommand{\Rmnum}[1]{\expandafter\@slowromancap\romannumeral #1@}

\allowdisplaybreaks

\numberwithin{equation}{section}
\begin{document}
	%%%%%%%%%%%%%%%%%%%%%%%%%%%%%%%%%%%%%%%%%%%%%%%%%%%%%%%%%%%%%%%%%%%%%%%%%%%%%%%%%%%%%%%%%%%%%%%%%%%%
	\title{New $\varepsilon$-regularity criteria of suitable weak solutions of   the 3D Navier-Stokes equations at one scale}
	\author{Cheng He\footnote{
 Division of Mathematics, Department of Mathematical \& Physical Sciences, National Natural Science Foundation of China, 100085,  P. R. China  Email:  hecheng@nsfc.gov.cn},\,\,\,\,Yanqing Wang\footnote{ Department of Mathematics and Information Science, Zhengzhou University of Light Industry, Zhengzhou, Henan  450002,  P. R. China Email: wangyanqing20056@gmail.com}\;~ and\,
	Daoguo Zhou\footnote{
	College of Mathematics and Informatics, Henan Polytechnic University, Jiaozuo, Henan 454000, P. R. China Email:
	zhoudaoguo@gmail.com }}
	\date{}
	\maketitle\vspace{-0.85cm}
	\begin{abstract}
In this paper,  by invoking the appropriate decomposition of pressure to exploit the energy hidden in pressure,  we present some new
		$\varepsilon$-regularity criteria for suitable weak solutions of the 3D  Navier-Stokes equations at one scale: for any $p,q\in [1,\infty]$ satisfying $1\leq 2/q+3/p <2$, there exists an absolute positive constant $\varepsilon$ such that $u\in L^{\infty}(Q(1/2))$ if
		$$\|u\|_{L^{p,q}(Q(1))}+\|\Pi\|_{L^{1 }(Q(1))}<\varepsilon.$$
		This  is an improvement of corresponding results recently  proved
		by Guevara   and Phuc in \cite[Calc. Var.  56:68, 2017]{[GP]}.
As an application of these $\varepsilon$-regularity criteria,    we improve the known   upper  box   dimension of the possible interior singular set of suitable weak solutions
               of the Navier-Stokes system from  $975/758(\approx1.286)$  \cite{[RWW]}
		to $2400/1903 (\approx1.261)$.
	\end{abstract}
	%%%%%%%%%%
	\section{Introduction}
	\label{intro}
	\setcounter{section}{1}\setcounter{equation}{0}
	We study  the following   incompressible Navier-Stokes equations in   three-dimensional space:
	\be\left\{\ba\label{NS}
	&u_{t} -\Delta  u+ u\cdot\ti
	u+\nabla \Pi=0, ~~\Div u=0,\\
	&u|_{t=0}=u_0,
	\ea\right.\ee
	where $u $ stands for the flow  velocity field, the scalar function $\Pi$ represents the   pressure.   The
	initial  velocity $u_0$ satisfies   $\text{div}\,u_0=0$.

	In this paper, we are concerned with the regularity   of suitable weak solutions  originated from Scheffer  in  \cite{[Scheffer1],[Scheffer2],[Scheffer3]}    to
	the 3D Navier-Stokes equations \eqref{NS}.
In comparison with the usual Leray-Hopf
weak solutions meeting with the energy inequality, the suitable weak solutions obey   local
energy inequality \eqref{loc}.	  A point is said to be a regular point to suitable weak solutions of the Navier-Stokes system \eqref{NS} as long as  $u$ is
 bounded in some neighborhood
of this point. The remaining  points will be called singular points and denoted by $\mathcal{S}$.
	In this direction,
	the celebrated Caffarelli-Kohn-Nirenberg theorem  involving the 3D Navier-Stokes equations is that   one
	dimensional Hausdorff measure of $\mathcal{S}$ is zero in \cite{[CKN]}.
	Roughly speaking, the
	regularity    of suitable weak solutions strongly rests on the so-called
	$\varepsilon$-regularity criteria (see,
	e.g., \cite{[CKN],[SS18],[GKT],[CY1],[Struwe],[HX],[Phuc],[RWW],[RS],[WW1],[CL],[GP],[Kukavica],[KP],[KY6],[WW2],[RWW],[WZ],[Vasseur],[Lin],[LS],[TX]}),  namely, some  scale-invariant quantities  with   sufficiently small at one scale or at all scales means the local regularity.
Before going further, we give some notations used throughout this paper.
For $q\in [1,\,\infty]$, the notation $L^{q}((0,\,T);X)$ stands for the set of measurable functions on the interval $(0,\,T)$ with values in $X$ and $\|f(t,\cdot)\|_{X}$ belongs to $L^{q}(0,\,T)$.
  For simplicity,   we write $$\|f\| _{L^{p,q}(Q(r))}:=\|f\| _{L^{q}(-r^{2},0;L^{p}(B(r)))}~~   \text{and}~~
  \|f\| _{L^{p}(Q(r))}:=\|f\| _{L^{p,p}(Q(r))}, $$
  where $Q(r)=B(r)\times (t-r^{2}, t)$ and $ B(r)$ denotes the ball of center $x$ and radius $r$.
  Now,
we  briefly recall some previous $\varepsilon$-regularity criteria: there exists an positive universal constant  $\varepsilon$ such that  	$u\in L^{\infty}(Q(1/2))$ provided one of the following conditions is satisfied
	\begin{itemize}
		\item Caffarelli, Kohn and Nirenberg \cite{[CKN]}
		\begin{equation}\label{CKN}
		\|u\|_{L^{3}(Q(1))}+\|u\Pi\|_{L^1(Q(1))}+\|\Pi\|_{L^{1,5/4}(Q(1))}
		\leq \varepsilon.
		\end{equation}
		\item Lin \cite{[Lin]}, Ladyzenskaja and  Seregin \cite{[LS]},
		\begin{equation}\label{Lin}
		\|u\|_{L^{3}(Q(1))}+\|\Pi\|_{L^{3/2}(Q(1))}\leq \varepsilon.
		\end{equation}
		\item Vasseur \cite{[Vasseur]}, for any $p>1$,
		\begin{equation}%\label{Vasssmall}
		\|u\|_{L^{2,\infty}(Q(1))} + \|\nabla u\|_{L^{2}(Q(1))}  +\|\Pi\|_{L^{1, p}(Q(1))}    \leq \varepsilon.
		\end{equation}
\item Wang and Zhang \cite{[WZ]}
		\be\label{WZ}
		\|u\|_{L^{2,\infty}(Q(1))} +\| u\|_{L^{4, 2} (Q(1))}+\|\Pi\|_{L^{2,1}(Q(1))}    \leq \varepsilon. \ee
		\item Choi   and Vasseur \cite{[CV]},
		Guevara and Phuc \cite{[GP]}
		\begin{equation}\label{CVGP}
		\|u\|_{L^{2,\infty}(Q(1))} + \|\nabla u\|_{L^{2}(Q(1))}  +\|\Pi\|_{L^{1}(Q(1))}   \leq \varepsilon.
		\end{equation}
				\item
		Guevara and Phuc \cite{[GP]}
		\be\label{GP}
		\|u\|_{L^{2p, 2q} (Q(1))}+\|\Pi\|_{L^{p, q}(Q(1))}\leq \varepsilon,
		~~~{3}/{p}+{2}/{q}=
		7/2~~~\text{with}~1\leq q\leq2.\ee
		A special case of $ q=2,p=6/5$ can be found in \cite{[Phuc]} by Phuc.
			\end{itemize}
It is worth pointing out that all  $\varepsilon$-regularity criteria mentioned above requires only one 	radius, namely, these criteria hold at one scale. In a very recent summary involving $\varepsilon$-regularity   criteria in \cite[page 8]{[SS18]} written by Seregin and \v{S}ver\'{a}k, a comment on regular criterion \eqref{Lin}, is:`` the bootstrapping enables to lower the exponent in the smallness condition from $3$ to $\f52+\delta$( at the cost of having to use smallness at all scales)."  For $\varepsilon$-regularity criteria   at all scales, we refer the reader to \cite{[GKT],[TX]}.
The  objective of this paper  is to establish a regular criterion   criteria for $\f52+\delta$ at one scale.
 	Our first result for   suitable weak solutions of the Navier-Stokes equations is stated as follows:	
	\begin{theorem}\label{mainthm}
		Let  the pair $(u,  \Pi)$ be a suitable weak solution to the 3D Navier-Stokes system \eqref{NS} in $Q(1)$.
		There exists an absolute positive constant $\varepsilon$
		such that if the pair $(u,\Pi)$ satisfies	\be\label{optical}\|u\|_{L^{p,q}(Q(1))}+\|\Pi\|_{L^{1}(Q(1))}<\varepsilon,~~1\leq 2/q+3/p <2, 1\leq p,\,q\leq\infty,\ee
		then, $u\in L^{\infty}(Q(1/2)).$
	\end{theorem}
	\begin{remark}
Note that    both \eqref{Lin} and \eqref{GP} belong to same kind in the sense $\|u\|_{L^{2q,2p}}+\|\Pi\|_{L^{q,p}}$. Hence, Theorem \ref{mainthm}   not only lowers the exponent of $\varepsilon$-regularity criteria  at one scale   but also  relaxes the integral condition
of pressure in \eqref{Lin} and \eqref{GP}.
	\end{remark}
	\begin{remark}
		Theorem \ref{mainthm} is a generalization of  \eqref{CKN}-\eqref{GP}. The pressure $\Pi$ in terms of $\nabla\Pi$ in equations \eqref{NS} allows us to replace $\Pi$ by $ \Pi-\overline{\Pi}_{B(1)} $ in \eqref{optical} as well as \eqref{CKN}-\eqref{GP}.  At present we are not  able to  prove \eqref{optical} for $2/p+3/q =2$ and this is still an open problem.
	\end{remark}
	\begin{remark}
		The method for Theorem  \ref{mainthm} presented here can be applicable to   suitable weak solutions of the  incompressible
magnetohydrodynamic equations \cite{[CY1],[HX]}. Here, we omit the detail.
	\end{remark}

It should be pointed out that the criterion  \eqref{Lin}
has been successfully applied to the investigation of the Navier-Stokes equations (see eg. \cite{[CKN],[GKT],[KY6],[Kukavica],[KP],[RWW],[RS],[WW2]}). An analogue of \eqref{Lin} in Theorem \ref{mainthm} is the following one
	\begin{coro}\label{maincoro}
		Assume that  the pair $(u,  \Pi)$ be a suitable weak solution to the 3D Navier-Stokes system \eqref{NS} in $Q(1)$.
		For each $\delta>0$, there exists an absolute positive constant $\varepsilon $
		such that  $u\in L^{\infty}(Q(1/2))$ provided	
\be\label{special}\|u\|_{L^{5/2+2\delta}(Q(1))}
+\|\Pi\|_{L^{5/4+\delta}(Q(1))}<\varepsilon.\ee	
	\end{coro}
Let us give some comments on the proof  of Theorem \ref{mainthm}.
Inspired by the argument in \cite{[GP]},
	the  idea to prove Theorem \ref{mainthm} is to establish
	an effective
	iteration scheme via local energy inequality \eqref{loc}.
The main   difficulty in the proof is that how to take full advantage of the left hand side (the energy) to control the right hand  side in the local energy inequality \eqref{loc}. Indeed,
  in \cite{[GP]}, the   term induced by convection term  $u\cdot\nabla u$ and pressure term $\nabla \Pi$ in local energy inequality \eqref{loc} is
\be\label{GP1} \int^{t}_{-T}\int_{\mathbb{R}^{3}} 2u\cdot\nabla\phi pdxds,\ee
where $p=\Pi+\f{1}{2}|u|^{2}$  is Bernoulli (total) pressure.
Guevara and Phuc bounded \eqref{GP1} by the energy and condition \eqref{GP}.
It seems that there exists at most one velocity $u$ controlled by the energy  in \eqref{GP1}.
%There are two
%	crucial ingredients in the proof of Theorem \ref{mainthm}.
To prove theorem \ref{mainthm},
we rewrite \eqref{GP1} as
\be\label{GP2} \int^{t}_{-T}\int_{\mathbb{R}^{3}} u\cdot\nabla\phi|u|^{2}dxds+ \int^{t}_{-T}\int_{\mathbb{R}^{3}} 2u\cdot\nabla\phi\Pi dxds:=I+II,\ee
First, to estimate the  term $I$,
	we establish the following inequality
	\be
	\|u\|^{3}_{L^{3}(Q(1/2))}\leq C 2^{-\f{3(\alpha-1)}{2}} \|u\|_{L^{p,q}(Q(1/2))}^{\alpha}\Big(\|u\|_{L^{2,\infty}(Q(1/2))}^{2}+\|\nabla u\|_{L^{2}(Q(1/2))}^{2}\Big)^{(3-\alpha)/2},\label{zwzc}
	\ee
	where $\alpha$ is  defined in \eqref{alpha}. Inequality \eqref{zwzc} allows us at most two velocity $u$ controlled by the energy. Here, we need
$\alpha>1$ to  apply the iteration Lemma \ref{iter}.
The second key point is to exploit the energy hidden in  the pressure  to bound the term $II$.
As in
  \eqref{GP1},   term $II$ contains one velocity $u$  at the first sight.
It is worth remarking that the corresponding  elliptic equation of pressure $\Pi$ and that of Bernoulli  pressure $p$ is completely different. Fortunately, we can treat with $II$ by the energy and hypothesis \eqref{optical}.
 To this end,
	by choosing appropriate test function in local energy inequality, we
	utilize the decomposition of pressure to split the term $\|u\Pi\|_{L^{1}(Q(1/2))}$ into three parts: $\Pi_1$ is in terms of $u$ bounded by the Calder\'on-Zygmund theorem;    $\Pi_2$ involving $u$ is a harmonic function;   $\Pi_3$ depending on
 $\Pi$ is also a harmonic function. These three parts are controlled   separately
 (see Lemma 2.1). In summary, there holds
	$$\ba\|u\Pi\|_{L^{1}(Q(1/2))}&\leq   \|\Pi_1 u\|_{L^{1}(Q(1/2))}  +\|\Pi_2 u \|_{L^{1}(Q(1/2))}+\|\Pi_3 u \|_{L^{1}(Q(1/2))}\\&\leq C \|u\|^{3}_{L^{3}(Q(1))}+
C\|\Pi\|_{L^{1}(Q(1))}\|u\|_{L^{2,\infty}(Q(1))}.\ea$$
This decomposition
of pressure helps us make full use of the energy and lower the
integrability for pressure $\Pi$ in the space-time variable.
 Finally,   the new local energy bounds  can be  derived
	\be
	\|u\|^2_{L^{2,\infty}(Q(1/2))} + \|\nabla u\|^2_{L^{2}(Q(1/2))} \leq  C   \|u\|^{2}_{L^{p,q}(Q(1))}+C \|u\|^{2\alpha/(\alpha-1)}_{L^{p,q}(Q(1))}+C \|\Pi\|^{2}_{L^{1}(Q(1))}.\label{leb}
	\ee
The proof of
Theorem \ref{mainthm} is an immediate consequence of the above inequality and  \eqref{CVGP}.

	As an  application of Corollary \ref{maincoro}, we  refine the upper box-counting dimension of the possible
	interior singular set of suitable weak solutions to the 3D Navier-Stokes equations.
	Before starting the statement of our results, let us sketch the known works.
	There are several works \cite{[Kukavica],[KP],[RS],[WW2],[RWW],[KY6]} trying to show that
	the upper  box dimension of the
	singular   set of   suitable weak solutions of the 3D Navier-Stokes system is    at most 1 since the
Hausdorff dimension of a set is less than its  the upper
box dimension(see eg. \cite{[Falconer]}).  In two works \cite{[Kukavica],[KP]}, by the backward heat
kernel, Kukavica and his co-author Pei  proved that    this dimension   is less than or equal to $135/82(\approx1.646)$ and  $45/29(\approx1.552)$, respectively. This improved
	 	Robinson and
	Sadowski's \cite{[RS]} result $5/3(\approx1.667)$.
	Very recently, Koh and Yang introduced  a new and efficient iteration  approach to calculate the box-dimension in \cite{[KY6]}, where they
	proved that the fractal upper box dimension of $\mathcal{S}$ is bounded by $95/63(\approx1.508)$.  Shortly afterwards, inspired by Koh and Yang's work,
	authors in \cite{[WW2]} shown that  this dimension is at most $360/277(\approx 1.300)$. It should be noted that   a same tool \eqref{Lin} was employed   in \cite{[WW2],[KY6],[RS],[KP]}. Very recently, the authors in \cite{[RWW]} lower this dimension to $975/758(\approx1.286)$ via Guevara and Phuc's criterion \eqref{GP} for $p=q=10/7$. Here,  our result reads  below
	\begin{theorem}\label{boxthm}
		The upper box dimension of  $\mathcal{S}$ in \eqref{NS} is at most $2400/1903(\approx1.261).$
	\end{theorem}
	\begin{remark}
		This theorem is an improvement of the  known   box dimension of   $\mathcal{S}$   in \cite{[KP],[RS],[Kukavica],[KY6],[WW2],[RWW]}.
	\end{remark}
	\begin{remark}
		 The result in  \cite{[RWW]} and Theorem \ref{boxthm}  suggest that the new $\varepsilon$-regularity criteria  	yield better		upper box dimension of the
	singular   set in the Navier-Stokes equations.
	\end{remark}
	\begin{remark}

  To make the paper more readable, we will apply \eqref{special} for $\delta=0$ to obtain Theorem \ref{boxthm}.  In
contrast with the work of \cite{[WW2],[KY6],[RWW]}, the proof here 	requires new scaling invariant quantities  and associated  decay estimates.
\end{remark}
	By Vitali cover lemma and contradiction  arguments as in \cite{[WW2],[KP]}, Theorem  \ref{boxthm} turns out to be a  consequence of the following proposition.
	\begin{prop}\label{boxprop}
		Suppose that the pair $(u, \,\Pi)$ is a suitable weak solution to (\ref{NS}).  Then, for any $\gamma <2315/5709$, $(x ,\,t )$ is a regular point  provided there exists a sufficiently small universal positive constant $\varepsilon_{1}$ and $0<r<1$ such that
		\be\label{cond}\ba   \iint_{ Q (r)}
		|\nabla u |^{2} +| u |^{ 10/3}+|\Pi-\overline{\Pi}_{ B (r)} |^{ 5/3}+
		|\nabla \Pi| ^{5/4}dxds \leq    r^{5/3-\gamma}\varepsilon_{1}.  \ea\ee
	\end{prop}	
	\begin{remark}
		Proposition \ref{boxprop} is an improvement of corresponding results
		obtained    in \cite{[KP],[WW2],[RWW]}.
	\end{remark}
	\begin{remark}
Compared with the regularity criteria \eqref{CKN}-\eqref{special},
  the regularity condition \eqref{cond} is not scale invariant.
  Note that $ r^{-1}\iint_{Q(r)}|\nabla u |^{2}dxds, r^{-5/3}\iint_{Q(r)}\big(| u |^{ 10/3}+|\Pi |^{ 5/3}\big)dxds$ and $ r^{-5/4}
		\iint_{Q(r)}\nabla\Pi dxds$ are dimensionless quantities, hence, it seems that $(5/3-\gamma)(2400/1903)$ in \eqref{cond} can be seen as an interpolation   between  $1$, $5/4$ and $5/3$.
	\end{remark}

	The remainder of this paper is divided into  four sections.
	In Section  2, we first recall the definitions  of the upper box-counting dimension and suitable weak solutions to the Navier-Stokes equations. Then, we   present the decomposition of pressure
	and establish   some crucial  bounds for the scaling   invariant quantities.
	The third section is devoted to the proof of Theorem \ref{mainthm}.
	Section 4  is concerned with
	the box-counting dimension
	of the possible  singular set of suitable weak
	solutions.
	
	\noindent
	{\bf Notations:} Throughout this paper, the classical Sobolev norm $\|\cdot\|_{H^{s}}$  is defined as   $\|f\|^{2} _{{H}^{s}}= \int_{\mathbb{R}^{n}} (1+|\xi|)^{2s}|\hat{f}(\xi)|^{2}d\xi$, $s\in \mathbb{R}$.
	We denote by  $ \dot{H}^{s}$ homogeneous Sobolev spaces with the norm $\|f\|^{2} _{\dot{H}^{s}}= \int_{\mathbb{R}^{n}} |\xi|^{2s}|\hat{f}(\xi)|^{2}d\xi$.
			Denote
	the average of $f$ on the set $\Omega$ by
	$\overline{f}_{\Omega}$. For convenience,
	$\overline{f}_{r}$ represents  $\overline{f}_{B(r)}$.
		$|\Omega|$ represents the Lebesgue measure of the set $\Omega$. We will use the summation convention on repeated indices.
	$C$ is an absolute constant which may be different from line to line unless otherwise stated in this paper.

	%\section{Box dimension of singular set of  the Navier-Stokes system}\label{sec2}
	
	\section{Preliminaries}
	First, we begin with the \wblue{definitions} of the upper box-counting dimension and suitable weak solutions of Navier-Stokes equations \eqref{NS}, respectively.
	\begin{definition}\label{defibox}
		The  upper  box  dimension of a set $X$ is usually defined as		$$d_{\text{box}}(X)=\limsup_{\varepsilon\rightarrow0}\f{\log N(X,\,\varepsilon)}{-\log\varepsilon},$$
		where $N(X,\,\varepsilon)$ is the minimum number of balls of radius $\varepsilon$ required to cover $X$.
	\end{definition}
 Materials on     box  dimension and Hausdorff dimension  can
be found in  \cite{[Falconer]}.

	\begin{definition}\label{defi}
		A  pair   $(u, \,\Pi)$  is called a suitable weak solution to the Navier-Stokes equations \eqref{NS} provided the following conditions are satisfied,
		\begin{enumerate}[(1)]
			\item $u \in L^{\infty}(-T,\,0;\,L^{2}(\mathbb{R}^{3}))\cap L^{2}(-T,\,0;\,\dot{H}^{1}(\mathbb{R}^{3})),\,\Pi\in
			L^{3/2}(-T,\,0;L^{3/2}(\mathbb{R}^{3}));$\label{SWS1}
			\item$(u, ~\Pi)$~solves (\ref{NS}) in $\mathbb{R}^{3}\times (-T,\,0) $ in the sense of distributions;\label{SWS2}
			\item$(u, ~\Pi)$ satisfies the following inequality, for a.e. $t\in[-T,0]$,
			\begin{align}
				&\int_{\mathbb{R}^{3}} |u(x,t)|^{2} \phi(x,t) dx
				+2\int^{t}_{-T}\int_{\mathbb{R} ^{3 }}
				|\nabla u|^{2}\phi  dxds\nonumber\\ \leq&  \int^{t}_{-T }\int_{\mathbb{R}^{3}} |u|^{2}
				(\partial_{s}\phi+\Delta \phi)dxds
				+ \int^{t}_{-T }
				\int_{\mathbb{R}^{3}}u\cdot\nabla\phi (|u|^{2} +2\Pi)dxds, \label{loc}
			\end{align}
			where non-negative function $\phi(x,s)\in C_{0}^{\infty}(\mathbb{R}^{3}\times (-T,0) )$.\label{SWS3}
		\end{enumerate}
	\end{definition}
	%First, we also present the definition of suitable weak solution to the stationary case.

Now, we present the decomposition of the pressure $\Pi$, which plays an important role  in the proof of Theorem \ref{mainthm}.
		\begin{lemma}\label{lem2}
Denote  the standard normalized fundamental solution of Laplace equation by $\Gamma$  and suppose that $0<r<\rho<\infty$. Let $\eta\in C^{\infty}_{0}(B(\rho))$ such that $0\leq\eta\leq1$ in $B(\rho)$, $\eta\equiv1$ in $B(\f{r+3\rho}{4})$  and $|\nabla^{k}\eta |\leq C/(\rho-r)^{k}$ with $k=1,2$ in $B(\rho)$.  Then we can decompose pressure $\Pi$ in \eqref{NS} as follows
			\be\label{decompose pk}
			\Pi(x):=\Pi_{1}(x)+\Pi_{2}(x)+\Pi_{3}(x), \quad x\in B(\f{r+\rho}{2}),
			\ee
			where
			$$\ba
			\Pi_{1}(x)=&-\partial_{i}\partial_{j}\Gamma \ast (\eta (u_{j}u_{i})) ,\\
			\Pi_{2}(x)
			=&2\partial_{i}\Gamma \ast(\partial_{j}\eta(u_{j}u_{i}))- \Gamma \ast
			(\partial_{i}\partial_{j}\eta u_{j}u_{i}), \\
			\Pi_{3}(x)
			=&2\partial_{i}\Gamma \ast(\partial_{i}\eta \Pi)-\Gamma \ast(\partial_{i}\partial_{i}\eta \Pi).
			\ea
			$$
 Moreover, we have the following estimates
			\begin{align}
				& \| \Pi_1\|_{L^{3/2}(Q(\f{r+\rho}{2}))}
				\leq C\|  u\|^{2}_{L^{3}(Q(\rho))};\label{p1estimate}\\
				& \|  \Pi_2\|_{L^{3/2}(Q(\f{r+\rho}{2}))}
				\leq  \f{C\rho^{3}}{(\rho-r)^{3}}\|  u\|^{2}_{L^{3}(Q(\rho))};\label{p2estimate}\\
				& \|  \Pi_3\|_{L^{2,1}(Q(\f{r+\rho}{2}))}
				\leq \f{C\rho^{3/2 }}{(\rho-r)^{3}}\|\Pi\|_{L^{1}(Q(\rho))}.\label{p3estimate}
			\end{align}
		\end{lemma}
		\begin{proof}
			Thanks to  $\partial_{i}\partial_{i} \Pi=-\partial_{i}\partial_{j} (u_{i}u_{j})$ and Leibniz's formula, we conclude that
			\[\partial_{i}\partial_{i}(\Pi\eta)=-\eta \partial_{i}\partial_{j}(u_{j}u_{i})+
			2\partial_{i}\eta\partial_{i}\Pi+\Pi\partial_{i}\partial_{i}\eta.\]
This enables us to 		write,	for \wblue{$x\in B(\f{r+3\rho}{4})$,}
			\be\ba\label{p}
			\Pi(x)=&\Gamma \ast (-\eta \partial_{i}\partial_{j}(u_{j}u_{i})+			2\partial_{i}\eta\partial_{i}\Pi+\Pi\partial_{i}\partial_{i}\eta)\\
			=&-\partial_{i}\partial_{j}\Gamma \ast (\eta (u_{j}u_{i}))\\
			&+2\partial_{i}\Gamma \ast(\partial_{j}\eta(u_{j}u_{i}))-\Gamma \ast
			(\partial_{i}\partial_{j}\eta u_{j}u_{i})\\
			&\wblue{-2\partial_{i}\Gamma \ast(\partial_{i}\eta \Pi)}-\Gamma \ast(\partial_{i}\partial_{i}\eta \Pi)\\
			:= &\wblue{\Pi_1(x)+\Pi_2(x)+\Pi_3(x),}
			\ea\ee
			where  we have used integrating by parts.
	
The classical Calder\'on-Zygmund theorem ensures that
			$$\|\Pi_1\|_{L^{3/2}(B(\f{r+\rho}{2}))}
			\leq C\|  u\|^{2}_{L^{3}(B(\rho))}.$$
			In view of  the property of the cut-off function, there is no singularity in $\Pi_2$ and $\Pi_3$,
			in  $B(\f{r+\rho}{2})$.
			So, a straightforward computation gives
			$$
			|\Pi_2(x)|\leq \f{C }{(\rho-r)^{3}}\int_{B(\rho)}|u(y)|^{2}dy,\quad
			x\in  B(\f{r+\rho}{2}),$$
			and
			$$
			|\Pi_3(x)|\leq \f{C }{(\rho-r)^{3}}\int_{B(\rho)}|\Pi(y)|dy,\quad x\in B(\f{r+\rho}{2}).
			$$
			Then applying H\"older's inequality gives
			the desired estimates \eqref{p1estimate}, \eqref{p2estimate}  and \eqref{p3estimate}. This completes the proof of the lemma.
		\end{proof}
Before we turn our attentions to the proof of the inequality \eqref{zwzc}, we have to introduce
  \be\label{alpha}
\alpha=\f{2}{\f{3}{p}+\f{2}{q}}>1.
\ee
We refer the readers  to  \cite{[SS],[WZ]} for slightly different versions of Lemma 2.2.
\begin{lemma}\label{zc} Let $1\leq 2/q+3/p <2, 1\leq p,\,q\leq\infty $ and $\alpha$ be defined as above. There is an absolute constant $C$   such that
\be\ba\label{zw}
\|u\|_{L^{3}(Q(\rho))}^{3} \leq C \rho^{3(\alpha-1)/2} \|u\|_{L^{p,q}(Q(\rho))}^{\alpha}\Big(\|u\|_{L^{2,\infty}(Q(\rho))}^{2}+\|\nabla u\|_{L^{2}(Q(\rho))}^{2}\Big)^{(3-\alpha)/2}.
\ea\ee
\end{lemma}
\begin{proof}
First, we recall an interpolation inequality.
For each $2\leq l\leq\infty$ and $2\leq k\leq6$ satisfying $\f{2}{l}+\f{3}{k}=\f{3}{2}$, according to the H\"{o}lder   inequality  and the Young  inequality, we know that
\begin{align}
\|u\|_{L^{ k,l}(Q(\mu))}&\leq C\|u\|_{L^{2,\infty}(Q(\mu))}^{1-\f {2} {l}}\|u\|_{L^{6,2}(Q(\mu))}^{\f {2} {l}}\nonumber\\
&\leq C\|u\|_{L^{2,\infty} (Q(\mu))}^{1-\f {2} {l}}(\|u\|_{L^{2,\infty} (Q(\mu))}
+\|\nabla u\|_{L^2(Q(\mu))})^{\f {2} {l}}\nonumber\\
&\leq C (\|u\|_{L^{2,\infty} (Q(\mu))}
+\|\nabla u\|_{L^2(Q(\mu))}).\label{sampleinterplation}
\end{align}
It is clear that  $q/\alpha\geq1$  and $p/\alpha\geq1$. Let
 $(\f{p}{\alpha})^{\ast}$ and $(\f{q}{\alpha})^{\ast}$ be the
H\"older dual of $\f{p}{\alpha}$ and $\f{q}{\alpha}$.
An elementary computation gives that $2\leq2(\f{p}{\alpha})^{\ast}\leq6$
and
$$
\f{3}{2(\f{p}{\alpha})^{\ast}}+\f{2}{2(\f{q}{\alpha})^{\ast}}=\f{3}{2}.$$ Hence,
taking advantage
  of  H\"older's inequality and
\eqref{sampleinterplation}, we have
$$\ba
\iint_{Q(\rho)}|u|^{3}dxdt&=\iint_{Q(\rho)}|u|^{\alpha}|u|^{3-\alpha}dxdt\\
&\leq\|u\|_{L^{p,q}(Q(\rho))}^{\alpha}\|u\|^{3-\alpha}
_{L^{(3-\alpha)(\f{p}{\alpha})^{\ast},(3-\alpha)(\f{q}{\alpha})^{\ast}}(Q(\rho))}
\\
&\leq C\rho^{3(\alpha-1)/2}\|u\|_{L^{p,q}(Q(\rho))}^{\alpha}\|u\|^{3-\alpha}
_{L^{2(\f{p}{\alpha})^{\ast},2(\f{q}{\alpha})^{\ast}}(Q(\rho))}
\\
&\leq C \rho^{3(\alpha-1)/2} \|u\|_{L^{p,q}(Q(\rho))}^{\alpha}\Big(\|u\|_{L^{2,\infty}(Q(\rho))}^{2}+\|\nabla u\|_{L^{2}(Q(\rho))}^{2}\Big)^{(3-\alpha)/2}.
\ea$$
The proof of this lemma is   completed.
\end{proof}

For the convenience of the reader, we recall the following well-known iteration lemma.		
		\begin{lemma}\label{iter}\cite[Lemma V.3.1,   p.161 ]{[Giaquinta]}
			Let $I(s)$ be a bounded nonnegative function in the interval $[r, R]$. Assume that for every $\sigma, \rho\in [r, R]$ and  $\sigma<\rho$ we have			$$I(\sigma)\leq A_{1}(\rho-\sigma)^{-\alpha_{1}} +A_{2}(\rho-\sigma)^{-\alpha_{2}} +A_{3}+ \ell I(\rho)$$
			for some non-negative constants  $A_{1}, A_{2}, A_{3}$, non-negative exponents $\alpha_{1}\geq\alpha_{2}$ and a parameter $\ell\in [0,1)$. Then there holds
			$$I(r)\leq c(\alpha_{1}, \ell) [A_{1}(R-r)^{-\alpha_{1}} +A_{2}(R-r)^{-\alpha_{2}} +A_{3}].$$
		\end{lemma}
Note that if the pair $\big(u(x,t),\Pi(x,t)\big)$ is a solutions of \eqref{NS}, then, for any $\lambda>0$, the pair $\big(\lambda u(\lambda x,\lambda^{2}t),\lambda^{2}\Pi(\lambda x,\lambda^{2}t)\big)$ also solves \eqref{NS}. Hence, as in \cite{[CKN]},
	we introduce the following dimensionless quantities:
	\be \begin{aligned}&E( r)=r^{-1}\|u\|^{2}_{L^{2,\infty} (Q(r))},
		&E&_{\ast}( r)=r^{-1}\|\nabla u\|^{2}_{L^{2}(Q(r))}
		,  \nonumber\\
		&E_{p}( r)=r^{p-5}\|u\|^{p}_{L^{p}(Q(r))}
		,&P&_{5/4}( r)= r^{-5/4}
		\|\nabla \Pi\|^{5/4}_{L^{5/4}(Q(r))},\nonumber
		\\
		& P_{5/4} ( r )= r^{ -5/2}
		\Big\|\Pi-\overline{\Pi}_{B(r)}\Big\|^{5/4}_{L^{5/4}(Q(r))},
		&P&_{5/3} ( r )= r^{ -5/3}		\Big\|\Pi-\overline{\Pi}_{B(r)}\Big\|^{5/3}_{L^{5/3}(Q(r))}.
		\nonumber
		\ea\ee	
In the spirit of \cite{[RWW]}, we derive some decay estimates involving the scaling invariant quantities, which are helpful in the proof of Proposition \ref{boxprop}.	
		\begin{lemma}\label{lemma2.1}
			For $0<r\leq\f{1}{2}\rho$ and $8/3\leq b\leq6,$
			there is an absolute constant $C$  independent of  $r$ and $\rho$,~ such that
			\begin{align}
				&E_{5/2}( r)\leq  C \left(\dfrac{\rho}{r}\right)^{ 5/4}
				E^{\f{2b-5}{2(b-2)}}( \rho)E^{\f{b}{4(b-2)}}_{\ast}(  \rho)
				+C\left(\dfrac{r}{\rho}\right)
				^{5/2}E^{5/4}( \rho).\label{inter3}
			\end{align}
		\end{lemma}
		\begin{proof}
			Taking advantage of  the H\"older  inequality    and the Poincar\'e-Sobolev  inequality, for any $5/2<b\leq6$, we see that
			\begin{align}
				\int_{B(r)}|u-\bar{u}_{B(\rho)}|^{5/2}dx\leq&
				C\B(\int_{B(r)}|u-\bar{u}_{B(\rho)}|^{2}dx\B)
				^{\f{2b-5}{2(b-2)}}
				\B(\int_{B(r)}|u-\bar{u}_{B(\rho)} |
				^{b}dx\B)^{\f{1}{2(b-2)}}\nonumber\\
				\leq&Cr^{\f{ (6-b)}{4(b-2)}}\B(\int_{B(\rho)}|u|^{2}dx\B)
				^{\f{2b-5}{2(b-2)}}
				\B(\int_{B(\rho)}|\nabla u|^{2}dx\B)^{\f{b}{4(b-2)}}. \nonumber
			\end{align}
By means of  the triangle inequality and the last inequality, we know that
			\begin{align}
				\int_{B(r)}|u|^{5/2}dx \leq& C\int_{B(r)}|u-\bar{u}_{{\rho}}|^{5/2 }dx
				+C\int_{B(r)}|\bar{u}_{{\rho}}|^{5/2} dx \nonumber\\
				\leq& Cr^{\f{ (6-b)}{4(b-2)}}\B(\int_{B(\rho)}|u|^{2}dx\B)
				^{\f{2b-5}{2(b-2)}}
				\B(\int_{B(\rho)}|\nabla u|^{2}dx\B)^{\f{b}{4(b-2)}} \nonumber\\&+
				\f{r^{3} C}{\rho^{\f{15}{4}}}\B( \int_{B(\rho)}|u|^{2}dx\B)^{5/4}.\nonumber
			\end{align}
			Integrating with respect to $s$ from $t-\mu^{2}$ to $t$
  and utilizing the H\"older  inequality again, for any $b\geq8/3$, we get
			\begin{align}
				\iint_{Q(r)}|u|^{5/2}dxds
				\leq& Cr^{\f{5}{4} }\B(\sup_{t-\rho^{2}\leq s\leq t}\int_{B(\rho)}|u |^{2}dx\B)
				^{\f{2b-5}{2(b-2)}}
				\B(\iint_{Q(\rho)}|\nabla u|^{2}dxds\B)^{\f{b}{4(b-2)}} \nonumber\\&+
				C\f{r^{5}}{\rho^{\f{15}{4}}}\B(\sup_{t-\rho^{2 }\leq s\leq t}\int_{B(\rho)}|u|^{2}dx\B)^{5/4},
			\end{align}
Therefore,
			$$
			E_{5/2}( r ) \leq  C \left(\dfrac{\rho}{r}\right)^{5/4}
			E^{\f{2b-5}{2(b-2)}}( \rho )E^{\f{b}{4(b-2)}}_{\ast}( \rho )
			+C\left(\dfrac{r}{\rho}\right)
			^{5/2}E^{5/4}( \rho ).
			$$
					\end{proof}
		In
		the spirit of   [17, Lemma 2.1, p.222], we can apply the interior estimate of harmonic function to establish the following   decay estimate of pressure   $\Pi-\overline{\Pi}_{ B (r)}$.
	Since	the pressure  $\Pi$ is in terms of
		$\nabla \Pi$  in equations \eqref{NS}, as said before, we can employ
		this lemma in the proof of Theorem \ref{boxthm} and Proposition \ref{boxprop}.
		\begin{lemma}\label{presure}
			For $0<r\leq\f{1}{8}\rho$, there exists an absolute constant $C$  independent of $r$ and $\rho$ such that
			\begin{align}
				P_{5/4}( r)
				\leq C\left(\f{\rho}{r}\right)^{5/2}E_{5/2}( \rho)+
				C\left(\f{r}{\rho}\right)^{7/4}P_{5/4}( \rho).\label{ppp} \end{align}
		\end{lemma}
		\begin{proof}
Fix a smooth function $\phi$ supported in $B(\rho/2)$ and with value 1 on the ball $B(\f{3}{8}\rho)$. Moreover, there holds   $0\leq\phi\leq1$ and
			$|\nabla\phi |\leq C\rho^{-1},~|\nabla^{2}\phi |\leq
			C\rho^{-2}.$
			
As in Lemma \ref{lem2}, we have
			$$
			\partial_{i}\partial_{i}(\Pi\phi)=-\phi \partial_{i}\partial_{j}\big[u_{i}u_{j}\big]
			+2\partial_{i}\phi\partial_{i}\Pi+\Pi\partial_{i}\partial_{i}\phi
			.$$
			For any $x\in B(\f{3}{8}\rho)$, we deduce from integrations by parts that
			\begin{align}
				\Pi(x)=&\Gamma \ast \{-\phi \partial_{i}\partial_{j}[u_{i}u_{j}]
				+2\partial_{i}\phi\partial_{i}\Pi+\Pi\partial_{i}\partial_{i}\phi
				\}\nonumber\\
				=&-\partial_{i}\partial_{j}\Gamma \ast (\phi [u_{i}u_{j}])\nonumber\\
				&+2\partial_{i}\Gamma \ast(\partial_{j}\phi[u_{i}u_{j}])-\Gamma \ast
				(\partial_{i}\partial_{j}\phi[u_{i}u_{j}])\nonumber\\
				& +2\partial_{i}\Gamma \ast(\partial_{i}\phi \Pi) -\Gamma \ast(\partial_{i}\partial_{i}\phi \Pi)\nonumber\\
				=: &\Pi_1(x)+\Pi_2(x)+\Pi_3(x),\label{pp}
			\end{align}
Thanks to  $\phi(x)=1$ ($x\in B(\rho/4$)), we discover that
			\[
			\Delta(\Pi_2(x)+\Pi_3(x))=0.
			\]
By virtue  of the interior  estimate of harmonic function
			and the  H\"older  inequality, we know that, for every
			$ x_{0}\in B(\rho/8)$,
			$$\ba
			|\nabla (\Pi_2+\Pi_3)(x_{0})|&\leq \f{C}{\rho^{4}}\|(\Pi_2+\Pi_3)\|_{L^{1}(B_{x_{0}}(\rho/8))}
			\\
			&\leq \f{C}{\rho^{4}}\|(\Pi_2+\Pi_3)\|_{L^{1}(B(\rho/4))}\\
			&\leq \f{C}{\rho^{17/5}}\|(\Pi_2+\Pi_3)\|_{L^{5/4}(B(\rho/4))}.
			\ea$$
Consequently,
			$$\|\nabla (\Pi_2+\Pi_3)\|^{5/4}_{L^{\infty}(B(\rho/8))}\leq C \rho^{-17/4}\|(\Pi_2+\Pi_3)\|^{5/4}_{L^{5/4}(B(\rho/4))}.$$
			This together with   the  mean value theorem gives that, for each $r\leq \f{1}{8}\rho$,
			$$\ba
			\|(\Pi_2+\Pi_3)-\overline{(\Pi_2+\Pi_3)}_{B(r)}\|^{5/4}_{L^{5/4}(B(r))}\leq&
			Cr^{3} \|(\Pi_2+\Pi_3)-\overline{(\Pi_2+\Pi_3)}_{B(r)}\|^{5/4}_{L^{\infty}(B(r))}\\
			\leq& C
			r^{17/4} \|\nabla (\Pi_2+\Pi_3)\|^{10/7}_{L^{\infty}(B(\rho/8))}\\
			\leq& C\Big(\f{r}{\rho}\Big)^{17/4}\|(\Pi_2+\Pi_3)\|^{5/4}_{L^{5/4}
				(B(\rho/4))}.
			\ea$$
Sine $(\Pi_2+\Pi_3)-\overline{(\Pi_2+\Pi_3)}_{B(\rho/4)}$ is also a harmonic function  on $B(\rho/4)$, we see that
			$$\ba
			&\|(\Pi_2+\Pi_3)-\overline{(\Pi_2+\Pi_3)}_{B(r)}\|
			^{5/4}_{L^{5/4}(B(r))}
			\\
			\leq & C\Big(\f{r}{\rho}\Big)^{17/4}
			\|(\Pi_2+\Pi_3)-\overline{(\Pi_2+\Pi_3)}_{B(\rho/4)}\|
			^{5/4}
			_{L^{5/4}(B(\rho/4))}.
			\ea$$
The triangle inequality implies that
			$$\ba
			&\|(\Pi_2+\Pi_3)-\overline{(\Pi_2+\Pi_3)}_{B(\rho/4)}\|_{L^{5/4}(B(\rho/4))}\\
			\leq& \|\Pi-\overline{\Pi}_{B(\rho/4)}\|_{L^{5/4}(B(\rho/4))}
			+\|\Pi_1-\overline{\Pi_1}_{B(\rho/4)}\|_{L^{5/4}(B(\rho/4))}
			\\
			\leq& C\|\Pi-\overline{\Pi}_{B(\rho)}\|_{\wblue{L^{5/4}(B(\rho/4))}}
			+C\|\Pi_1\|_{L^{5/4}(B(\rho/4))},
			\ea$$
			which means that
			\begin{align}
				&\|(\Pi_2+\Pi_3)-\overline{(\Pi_2+\Pi_3)}_{B(r)}\|
				^{5/4}_{L^{5/4}(B(r))}\nonumber\\
				\leq& C\Big(\f{r}{\rho}\Big)^{17/4}\Big(\|\Pi-\overline{\Pi}
				_{B(\rho)}\|^{5/4}_{L^{5/4}(B(\rho))}
				+\|\Pi_1\|^{5/4}_{L^{5/4}(B(\rho/4))}\Big).
				\label{p2rou}\end{align}
In view of the classical Calder\'on-Zygmund theorem, we thus infer that
			\begin{align}
				\int_{B(\rho/4)}|\Pi_1(x)|^{5/4}dx
				\leq  C \int_{B(\rho/2)}|u|^{5/2}dx.
				\label{lem2.4.2}
			\end{align}
This also yields, for any  $r\leq \f{1}{8}\rho$,
			\begin{align}
				\int_{B(r)}|\Pi_1(x)|^{5/4}dx \leq C \int_{B(\rho/2)}|u|^{5/2}dx.
				\label{lem2.4.3}\end{align}
Integrating in time  on $(t-r^{2}, t)$ and using the triangle inequality, we conclude using \eqref{p2rou}-\eqref{lem2.4.3}  that
			\begin{align}
				&\iint_{Q(r)}|\Pi-\overline{\Pi}_{B(r)}|^{5/4}dxds\nonumber\\
				\leq& \iint_{Q(r)}|\Pi_1-\overline{\Pi_1}_{B(r)}|^{5/4}dxds+				\iint_{Q(r)}|\Pi_2+\Pi_3-\overline{(\Pi_2+\Pi_3)}_{B(r)}|^{5/4}dxds
				\nonumber\\ \leq & C\iint_{Q(r)}|\Pi_1|^{5/4}dxds+C\Big(\f{r}{\rho}\Big)^{17/4}
\Big(\|\Pi-\overline{\Pi}
				_{B(\rho)}\|^{10/7}_{L^{5/4}(B(\rho))}
				+\|\Pi_1\|^{10/7}_{L^{5/4}(B(\rho/4))}\Big)
				\nonumber\\ \leq &C \iint_{Q(\rho/2)}|u|^{5/2}dxds+C\Big(\f{r}{\rho}\Big)^{17/4}
				\|\Pi-\overline{\Pi}
				_{B(\rho)}\|^{5/4}_{L^{5/4}(B(\rho))},\label{pressuresan}
			\end{align}
			which means \wblue{\eqref{ppp}}.
			The proof of this lemma is completed.
		\end{proof}

		\section{Proof  of Theorem \ref{mainthm}}
		\label{sec3}
		\setcounter{section}{3}\setcounter{equation}{0}
The goal of this section is to prove Theorem \ref{mainthm}.
As said before, it is enough   to show \eqref{leb}.
We state precise proposition 	involving  local energy energy bound \eqref{leb} below.
\begin{prop}\label{lebp} Let $\alpha$ be defined as in \eqref{alpha}. Suppose that $(u,\Pi)$ is a suitable weak solution
to the Navier-Stokes equations in  $Q(R)$. Then there holds, for any $R>0$
\be\ba \label{key ineq}
\|u\|^2_{L^{2,\infty}(Q(R/2))}+\|\nabla u\|^2_{L^{2}(Q(R/2))}& \\ \leq  C  R^{(3\alpha-4)/\alpha}& \|u\|^{2}_{L^{p,q}(Q(R))}\\&+CR^{(3\alpha-5)/(\alpha-1)} \|u\|^{2\alpha/(\alpha-1)}_{L^{p,q}(Q(R))}+CR^{-6}\|\Pi\|^{2}_{L^{1}(Q(R))}.
\ea\ee
\end{prop}
\begin{proof}Consider $0<R/2\leq r<\f{3r+\rho}{4}<\f{r+\rho}{2}<\rho\leq R$. Let $\phi(x,t)$ be non-negative smooth function supported in $Q(\f{r+\rho}{2})$ such that
$\phi(x,t)\equiv1$ on $Q(\f{3r+\rho}{4})$,
$|\nabla \phi| \leq  C/(\rho-r) $ and $
|\nabla^{2}\phi|+|\partial_{t}\phi|\leq  C/(\rho-r)^{2} .$

By means of H\"older's inequality, we arrive at
\be\ba\int^{t}_{-T }\int_{\mathbb{R}^{3}} |u|^{2}
 (\partial_{s}\phi+\Delta \phi)dxds&\leq \f{C}{(\rho-r)^{2}}\iint_{Q(\f{r+\rho}{2})} |u|^{2}
dxds\\&\leq  \f{C\rho^{5/3}}{(\rho-r)^{2}}\B(\iint_{Q(\rho)} |u|^{3}
dxds\Big)^{2/3}
\\&=:L_{1}.\label{firstterm}\ea\ee
Thanks to the local energy inequality \eqref{loc} and the decomposition of  pressure  in Lemma \ref{lem2}, we know that
\begin{align}\label{loc2}
 &\int_{B(\f{r+\rho}{2})} |u(x,t)|^{2} \phi(x,t) dx
 +2\iint_{Q(\f{r+\rho}{2})}
  |\nabla u|^{2}\phi  dxds  \leq  L_{1} +L_{2}+L_{3}+L_{4}+L_{5},
 \end{align}
where $$\ba L_{2}=\f{C}{(\rho-r)} \iint_{Q(\f{r+\rho}{2})}  |u|^{3} dxds;~\\
L_{3}=
\f{C}{(\rho-r)} \iint_{Q(\f{r+\rho}{2})} u\Pi_{1}dxds;~\\
L_{4}=
\f{C}{(\rho-r)} \iint_{Q(\f{r+\rho}{2})} u\Pi_{2}dxds;~\\
L_{5}=
\f{C}{(\rho-r)} \iint_{Q(\f{r+\rho}{2})} u\Pi_{3}dxds.~\ea$$
By the H\"older inequality and \eqref{p1estimate}-\eqref{p3estimate}, we find that
\begin{align}
&L_{3}\leq \f{C}{(\rho-r)}\|  \Pi_{1}\|_{L^{3/2}(Q(\f{r+\rho}{2}))}\|  u\|_{L^{3}(Q(\f{r+\rho}{2}))}\leq \f{C}{(\rho-r)}  \|u\|^{3}_{L^{3}(Q(\rho))},\label{thierdterm1}\\
&L_{4}\leq \f{C}{(\rho-r)}\|  \Pi_{2}\|_{L^{3/2}(Q(\f{r+\rho}{2}))}\|  u\|_{L^{3}(Q(\f{r+\rho}{2}))}\leq \f{C\rho^{3}}{(\rho-r)^{4}}  \|u\|^{3}_{L^{3}(Q(\rho))},\label{thierdterm2}\\
&L_{5}\leq \f{C}{(\rho-r)}\|  \Pi_{3}\|_{L^{2,1}(Q(\f{r+\rho}{2}))}\|  u\|_{L^{2,\infty}(Q(\f{r+\rho}{2}))}\leq \f{C\rho^{3/2}}{(\rho-r)^{4}} \|  \Pi\|_{L^{1} (Q(\rho))}  \|u\|_{L^{2,\infty}(Q(\rho)))}.\label{thierdterm3}
\end{align}
From \eqref{loc2}-\eqref{thierdterm3}, we see that it is enough to bound $\|u\|^{3}_{L^{3}(Q(\rho))}$. To this end,
plugging \eqref{zw} into \eqref{firstterm} \eqref{thierdterm1} and \eqref{thierdterm2} respectively, by the Young inequality, we conclude that
\begin{align}\nonumber &L_{1}\leq \f{C\rho^{3+2/\alpha}}{(\rho-r)^{6/\alpha}}\|u\|^{2}_{L^{p,q}(Q(\rho))} +\f{1}{5}\B(\|u\|_{L^{2,\infty}(Q(\rho))}^{2}+\|\nabla u\|_{L^{2}(Q(\rho))}^{2}\B),\\
\nonumber &L_{2} +L_{3}
\leq \f{C\rho^{3}}{(\rho-r)^{2/(\alpha-1)}}\|u\|^{2\alpha/(\alpha-1)}_{L^{p,q}(Q(\rho))}
+\f{1}{5}\B(\|u\|_{L^{2,\infty}(Q(\rho))}^{2}+\|\nabla u\|_{L^{2}(Q(\rho))}^{2}\B),\\
\nonumber
&L_{4}\leq \f{C\rho^{3(\alpha+1)/(\alpha-1)}}{(\rho-r)^{8/(\alpha-1)}}\|u\|^{2\alpha/(\alpha-1)}_{L^{p,q}(Q(R))}
+\f{1}{5}\B(\|u\|_{L^{2,\infty}(Q(\rho))}^{2}+\|\nabla u\|_{L^{2}(Q(\rho))}^{2}\B).
\end{align}
In addition, utilizing the Young inequality again, we get
$$L_{5}\leq   \f{C\rho^{3 }}{(\rho-r)^{8}} \|  \Pi\|^{2}_{L^{1} (Q(\rho))} +\f{1}{5} \|u\|^{2}_{L^{2,\infty}(Q(\rho)))}.$$
Collecting all the above estimates, we know that
$$\ba
&\|u\|_{L^{2,\infty}(Q(r))}^{2}+\|\nabla u\|_{L^{2}(Q(r))}^{2}\\\leq
&\f{C\rho^{(3\alpha+2)/\alpha}}{(\rho-r)^{6/\alpha}}\|u\|^{2}_{L^{p,q}(Q(\rho))} +\f{C\rho^{3}}{(\rho-r)^{2/(\alpha-1)}}\|u\|^{2\alpha/(\alpha-1)}_{L^{p,q}(Q(\rho))}\\
&+\f{C\rho^{3(\alpha+1)/(\alpha-1)}}{(\rho-r)^{8/(\alpha-1)}}\|u\|^{2\alpha/(\alpha-1)}_{L^{p,q}(Q(\rho))}
 +\f{C\rho^{3 }}{(\rho-r)^{8}} \|  \Pi\|^{2}_{L^{1} (Q(\rho))} +\f{4}{5}\B(\|u\|_{L^{\infty,2}(Q(\rho))}^{2}+\|\nabla u\|_{L^{2}(Q(\rho))}^{2}\B).
\ea$$
The  iteration Lemma \ref{iter} allows us to derive \eqref{key ineq} from the last inequality.
\end{proof}
		
		\section{Proof  of Proposition \ref{boxprop}}
		\label{sec4}
		\setcounter{section}{4}\setcounter{equation}{0}
		The main part of this sections is the proof of Proposition \ref{boxprop}.
		\begin{proof}[Proof of Proposition \ref{boxprop}]
Along the lines of \cite{[KY6],[WW2],[RWW]},
under the hypotheses of  \eqref{cond}, we select $2\rho<1$ such that $\rho^{\beta}<1/2$, where the parameter $\beta$ is to be
			determined  later   and
			\be\label{assume}
			\iint_{Q (2\rho)}
			|\nabla u |^{2} +| u |^{ 10/3}+|\Pi-\overline{\Pi}_{B(2\rho)} |^{ 5/3}+
			|\nabla \Pi| ^{5/4}dxds \leq    (2\rho)^{ 5/3-\gamma}\varepsilon_{1}.
			\ee
We will make use of the following result
			    \be E(\rho)\leq C\varepsilon_{1}^{3/5}\rho^{-\f{3\gamma}{5}},~~ (\gamma\leq5/12),\label{E}\ee      which is
			shown in \cite{[WW2]}. Here we omit its details, the reader is referred to \cite[Theorem 1.2, p.1768-1769]{[WW2]} for a proof.
			Second, iterating (\ref{ppp}) in Lemma \ref{presure}, we see  that
			\be\label{referee}
			P_{5/4}( \theta^{N}r)\leq C\sum^{N}_{k=1}\theta^{-\f{5}{2}+\f{7(k-1)}{4}}
			E_{5/2}( \theta^{N-k}r)+C\theta^{7N/4}
			P_{5/4}( r).
			\ee
In view of   the Poincar\'e-Sobolev  inequality and the H\"older  inequality, we know that
			\begin{align}
				\nonumber
				\|\Pi-\overline{\Pi}_{ B(r)}\|_{L^{5/4}(Q( r))} &\leq
				\|\Pi-\overline{\Pi}_{ B(r)}\|^{3/8}_{L^{5/4,15/7}(Q( r))} \|\Pi-\overline{\Pi}_{ B(r)}\|^{5/8}_{L^{5/4,1}(Q( r))}\nonumber\\
				&\leq
				Cr\|\nabla \Pi \|^{3/8}_{L^{5/4}(Q( r))} \|\Pi-\overline{\Pi}_{ B(r)}\|^{5/8}_{L^{5/3}(Q( r))},
				\label{pressureinti}\end{align}
			which in turn implies that
			$$
			P_{5/4}( r)\leq CP^{3/8}_{5/4}( r)
			P^{ 15/32}_{5/3}( r).
			$$
			Inserting this  inequality into \eqref{referee}, we have
			\be\label{refer}
			P_{5/4}( \theta^{N}r)\leq C\sum^{N}_{k=1}\theta^{-\f{5}{2}+\f{7(k-1)}{4}}
			E_{5/2}( \theta^{N-k}r)+C\theta^{7N/4}
			P^{3/8}_{5/4}( r)
			P^{ 15/32}_{5/3}( r).
			\ee
			Before going   further, we introduce some notations  $r=\rho^{\alpha}=\theta^{N}r$,~$ \theta=\rho^{\beta}$, ~$r_{i}=r =\theta^{-i}r=\rho^{\alpha-i\beta}(1\leq i\leq N)$, where $\alpha$ and $\beta$ are determined by $\gamma$.
			Their precise selections will be given
			in the end.
			As a consequence, by $E_{5/2}(u,r)\leq C\theta^{-\frac{5}{2}}E_{5/2}(u,\theta^{-1}r)$ and \eqref{refer}, we infer that
			\be\ba
			&P_{5/4}( r)+E_{5/2}( r) \\
			\leq&			C\sum^{N}_{{k=1}}\theta^{-\f{5}{2}+\f{7(k-1)}{4}}
			E_{5/2}( r_{k}) +C\theta^{7N/4}P^{3/8}_{5/4}
			( r_{N})
			P^{15/32}_{5/3}( r_{N})\\
			:=&\text{I}+\text{II}.\label{key}
			\ea\ee
In the light of \eqref{special}, it suffices to prove that there exists a constant $\sigma>0$ such that $P_{5/4}( \sigma)+E_{5/2}( \sigma)<\varepsilon_0$.
For this purpose,
			we   employ \eqref{inter3} with  $b=8/3$  in Lemma \ref{lemma2.1}, \eqref{E} and   \eqref{assume} to conclude
			$$\ba
			E_{5/2}( r_{k})&\leq C \Big(\f{\rho}{ r_{k}}\Big)^{\f{5}{4}}E^{1/4}( \rho)
			E_{\ast} ( \rho)+C\Big(\f{ r_{k}}{\rho}\Big)^{5/2}E^{5/4}( \rho)\\
			&\leq C\varepsilon_{1}^{3/4} \Big( \rho^{\f{23}{12}-\f{5}{4}(\alpha-k\beta)-\f{23\gamma}{20}}
			+\rho^{\f{5}{2}\alpha-\f{5}{2}-\f{5}{2}k\beta-\f{3\gamma}{4}} \Big).
			\ea$$
			Inserting  this inequality into I, we find that
			$$\ba
			I&\leq C\varepsilon_{1}^{6/7}\sum^{N}_{k=1}\B(\rho^{-\f{17\beta}{4}+  3k\beta -\f{5\alpha}{4}-\f{23\gamma}{20}+\f{23}{12}}
			+\rho^{-\f{17\beta}{4}+\f{5\alpha}{2}-\f{5}{2}-\f{3 k\beta}{4}-\f{3\gamma}{4}} \B).\ea $$
			To minimise the righthand side of this inequality, we choose
			\be\label{aerfa}
			\alpha=\f{4}{15}(3\beta +\f{53}{12}-\f{2\gamma}{5}+\f{3 N\beta}{4}).
			\ee
Hece, for sufficiently large $N$, there holds
			\be\ba \label{key1}
			\text{I} &\leq C\varepsilon_{1}^{6/7}\B(
			\rho^{-\f{5\beta}{4}-\f{5\alpha}{4}-\f{23\gamma}{20}+\f{23}{12}}
			+\rho^{-\f{17\beta}{4}+\f{5\alpha}{2}-\f{5}{2}-\f{3 N\beta}{4}-\f{3\gamma}{4}}\B)
			\\
			&\leq C\varepsilon_{1}^{6/7} \rho^{-\f{9\beta}{4}+\f{4}{9}
				-\f{61\gamma}{60}-\f{N\beta}{4}}.
			\ea \ee
To bound II, assume for a while there holds  $r_{N}\leq \rho$, that is
			\be
			\rho^{\alpha-N\beta}\leq\rho.\label{c3}
			\ee
Using the bounds \eqref{assume}
			and \eqref{aerfa}, we have the estimate
			\be\ba\label{key2}
			\text{II} \leq&  C\rho^{\f{7N\beta}{4}} r_{N}^{-\f{5}{4}}
			\B(\iint_{Q(r_{N})}|\nabla \Pi|^{5/4}dxds\B)^{3/8}
			\B( \iint_{Q(r_{N})}|\Pi-\overline{\Pi}_{2\rho}|^{5/3}dxds\B)^{15/32}\\
			\leq&  C\rho^{\f{7N\beta}{4}}r_{N}^{-\f{5}{4}}
			\B(\iint_{Q(2\rho)}|\nabla \Pi|^{5/4}dxds\B)^{3/8}
			\B(\iint_{Q(2\rho)}|\Pi-\overline{\Pi}_{2\rho}|^{5/3}dxds\B)^{15/32}\\
			\leq&  C\rho^{\f{11 N\beta }{4}-\f{19}{288} -\f{341\gamma}{480}-\beta}\varepsilon_{1}^{27/32}.
			\ea\ee
To gnarantee that $
			\text{I}+\text{II}
			\leq C\varepsilon_{1}^{3/4}\leq \varepsilon_0$,
			we need  $-\f{9\beta}{4}+\f{4}{9}
			-\f{61\gamma}{60}-\f{N\beta}{4}\geq 0$ and $\f{11 N\beta }{4}-\f{19}{288} -\f{341\gamma}{480}-\beta\geq 0$. In addition, we have derived
\eqref{key2} assuming that \eqref{c3}, so, we also require
  $\alpha-N\beta-1\geq0.$ Now, we collect all the restrictions of $\gamma$ below
			\be\label{last1}
			\gamma\leq \min\Big\{
			\f{5(16-9N\beta-81\beta)}{183},
			\f{5(2-9N\beta+9\beta)}{6},
			\f{5(792N\beta-19-288\beta)}{1023},\f{5}{12}
			\Big\}.
			\ee
			Maximising this bound on $\gamma$ with respect to $N\beta$, we discover that $N\beta=245/1903$. Furthermore, We deduce using \eqref{last1} that
			$$
			\beta=\f{245}{1903N}\leq\f{183}{405}\B(\f{2315}{5709}-\gamma\B).
			$$
			Hence, choosing $\beta$ sufficiently small by selecting $N$
			sufficiently large,
			we can have any
			$\gamma<2315/5709$.
			Then, we pick $\alpha=\f{4}{15}(3\beta -\f{2\gamma}{5}+\f{25766}{5709})$. We derive from
			  \eqref{key}, \eqref{key1} and \eqref{key2} that
			$$
			P_{10/7}(\sigma)+E_{20/7}(\sigma)
			\leq C\varepsilon_{1}^{3/4}< \varepsilon_0, $$
			with $\sigma=\rho^{\alpha}$. By \eqref{special}, finally, we see that $u \in L^{\infty}(Q(\sigma/2))$.
			This completes the proof of Proposition \ref{boxprop}.
		\end{proof}

		%\section{Proofs of Theorems \ref{the1.2}, \ref{the1.3} and \ref{the1.4}}
		%\label{sec4}

		%\section{Proofs of Theorems \ref{the1.2}, \ref{the1.3} and \ref{the1.4}}
		%\label{sec4}

		\section*{Acknowledgement}
		The research of Wang was partially supported by  the National Natural		Science Foundation of China under grant No. 11601492 and the
the Youth Core Teachers Foundation of Zhengzhou University of
Light Industry.
			The research of Zhou is
supported in part by the China Scholarship Council for one year study at Mathematical
Institute of University of Oxford and Doctor Fund of Henan Polytechnic University (No.
B2012-110).

	\end{document}